\documentclass[12pt]{amsart}

\usepackage{amssymb}
\usepackage{amsmath}
\usepackage{array}

\newcommand{\Lg}{\mbox{$\mathfrak g$}}

\newcommand{\Lk}{\mbox{$\mathfrak k$}}

\newcommand{\R}{\mbox{$\mathbf R$}}
\newcommand{\C}{\mbox{$\mathbf C$}}
\newcommand{\Q}{\mbox{$\mathbf H$}}

\newcommand{\cod}{\mbox{$\mbox{codim}\;$}}

\newcommand{\SU}[1]{\mbox{$\mathbf{SU}(#1)$}}
\newcommand{\U}[1]{\mbox{$\mathbf{U}(#1)$}}
\newcommand{\SP}[1]{\mbox{$\mathbf{Sp}(#1)$}}
\newcommand{\SO}[1]{\mbox{$\mathbf{SO}(#1)$}}
\newcommand{\OG}[1]{\mbox{$\mathbf{O}(#1)$}}
\newcommand{\Spin}[1]{\mbox{$\mathbf{Spin}(#1)$}}
\newcommand{\E}[1]{\mbox{$\mathbf{E}_{#1}$}}
\newcommand{\F}{\mbox{$\mathbf{F}_4$}}
\newcommand{\G}{\mbox{$\mathbf{G}_2$}}

\newcommand{\cref}[1]{Corollary~\ref{#1}}

\newcommand{\lref}[1]{Lemma~\ref{#1}}

\newtheorem{thm}{Theorem}[section]
\newtheorem*{thm*}{Theorem}
\newtheorem*{thmmain*}{MAIN THEOREM}
\newtheorem{lem}[thm]{Lemma}

\newtheorem*{prop*}{Proposition}

\theoremstyle{definition}

\theoremstyle{remark}

\title{Representations whose minimal reduction\\has a toric identity component}

\author{Claudio Gorodski and Alexander Lytchak}

\address{Instituto de Matem\'atica e Estat\'\i stica, Universidade de 
S\~ao Paulo, Rua do Mat\~ao, 1010, S\~ao Paulo, SP 05508-090, Brazil}
\email{gorodski@ime.usp.br}

\address{Mathematisches Institut, Universit\"at zu K\"oln, 
Weyertal 86-90, 50931 K\"oln, Germany}
\email{alytchak@math.uni-koeln.de}

\thanks{The first author was partially supported by the
CNPq grant 302472/2009-6 and the FAPESP project 2011/21362-2.}

\thanks{The second author was partially supported
by a Heisenberg grant of the DFG and by the SFB 878
{\it Groups, geometry and actions}}

\date{\today}

\begin{document}

\begin{abstract}
We classify irreducible representations of connected compact Lie
groups whose orbit space is isometric to the orbit space
of a representation of a finite extension of
(positive dimensional) toric group. They turn out to be exactly the non-polar
irreducible representations preserving an isoparametric submanifold
and acting with cohomogeneity one on it.
\end{abstract}

\maketitle

\section{Introduction}

A representation $\rho:G\to\OG V$ of a compact Lie group
on an Euclidean space~$V$ is called \emph{polar} if there exists a
subspace $\Sigma$, called a \emph{section}, that meets all $G$-orbits and meets them
always orthogonally~\cite{PT,BCO}. In this case, the normalizer
$N(\Sigma)$ acts on $\Sigma$ with kernel equal to the centralizer
$Z(\Sigma)$, the quotient group $\mathcal W=N(\Sigma)/Z(\Sigma)$ is finite, and the
inclusion $\Sigma\to V$ induces an isometry between orbit spaces
$\Sigma/\mathcal W=V/G$ with the quotient metrics. Conversely, assume $\rho:G\to\OG V$ is a representation
whose orbit space $X=V/G$ can also be given as the orbit space of a representation
of a finite group. On one hand, the set of regular points $X_{reg}$ is a
flat Riemannian manifold. On the other hand, due to O'Neill's formula, the horizontal
distribution of the Riemannian submersion $V_{reg}\to X_{reg}$ is integrable.
It follows that $\rho$ is polar~\cite{Ale,hlo}. Thus we have the following
characterization: \emph{a representation is polar
if and only if its orbit space is isometric to the orbit space
of a representation of a finite group.}

This paper follows the program introduced in~\cite{GL}, namely,
to hierarchize the representations of compact Lie groups in terms of the
complexity of their orbit spaces, viewed as metric spaces.
Dadok~\cite{D} has classified polar representations of connected groups
and showed that they are orbit-equivalent to the isotropy representations
of symmetric spaces. Herein we enlarge the class of polar
representations by replacing
``finite group'' by ``finite extension of toric group'' in the above
characterization.  Namely, we consider a class of representations defined
by the condition that the orbit space is isometric to the orbit space
of a representation of a finite extension of
toric group, and we classify such representations in the irreducible
case and for connected groups.

 It turns out that such representations have a remarkable connection with isoparametric submanifolds and polar representations. And, rather surprisingly,
only few families of non-polar representations can occur.
Our  result can be stated as follows.

\begin{thm}\label{main}
Let $\rho:G\to\OG{V}$  be an effective irreducible representation of a
connected compact Lie group $G$ on a Euclidean space $V$. Assume that
$\rho$ is not polar. Then the following conditions are equivalent:
\begin{enumerate}
\item[(a)] There is a representation $\tau:H\to\OG{W}$ of a compact
Lie group with toric identity component such that the orbit spaces
$V/G$, $W/H$ are isometric.
\item[(b)] There is a connected subgroup $\hat G$ of $\OG{V}$ containing $\rho(G)$
which acts polarly on $V$ with cohomogeneity one less than the cohomogeneity  of $\rho$.
\item[(c)] The action of $G$ leaves an isoparametric submanifold $S$ of $V$ invariant
and acts with cohomogeneity one on $S$.
\end{enumerate}
Moreover, the representations satisfying any
of the above conditions consist of three
disjoint families as follows:
\begin{enumerate}
\item[(i)] $\rho$ is one of the  non-polar irreducible representations of cohomogeneity
three:
\[ \begin{array}{|c|c|c|}
\hline
G & \rho & \textsl{Conditions} \\
\hline
\SO2\times\Spin9 & \R^2\otimes_{\mathbf R}\R^{16} & - \\
\U 2\times\SP n & \C^2\otimes_{\mathbf C}\C^{2n} &  n\geq2  \\
\SU2\times\SP n & S^2(\C^2)\otimes_{\mathbf H}\C^{2n} &  n\geq2 \\
\hline
\end{array} \]
\item[(ii)] The group $G$ is the semisimple factor of an irreducible polar representation
of Hermitian type such that action of $G$ is not orbit-equivalent to the polar representation:
\[ \begin{array}{|c|c|c|}
\hline
G & \rho & \textsl{Conditions} \\
\hline
\SU n & S^2\mathbf C^n & n\geq3 \\
\SU n & \Lambda^2\mathbf C^n & n=2p\geq6 \\
\SU n\times\SU n & \C^n\otimes_{\mathbf C}\C^n  & n\geq3 \\
\E6 & \mathbf C^{27} & - \\
\hline
\end{array} \]
\item[(iii)] $\rho$ is one of the two exceptions: $\SO3\otimes \G$, $\SO4\otimes\Spin7$.
\end{enumerate}
\end{thm}

A further motivation for Theorem~\ref{main}
comes from one of the main results of~\cite{GL}, which
gives another charaterization of the above class. Namely,
if $\rho:G\to\OG V$, $\rho':G'\to\OG{V'}$ are two non-polar
representations of a compact
Lie groups with isometric orbit spaces
$V/G=V'/G'$  such that  the restriction of $\rho$ to the identity
component $G^\circ$ is irreducible but that of $\rho'$ to $(G')^\circ$ is
reducible,
then the restriction of $\rho$ to $G^\circ$ satisfies the condition~(a),
for some representation $\tau:H\to\OG W$.

We are grateful to Wolfgang Ziller for pointing
out the connection with the work of Kollross
that considerably simplifies our classification.

\section{Non-classifying arguments}

Since principal orbits of polar representations are isoparametric submanifolds of Euclidean space~\cite{PT},
(b) implies~(c). Conversely, assume (c) holds. Then $S$ is full and irreducible, thus
either it is homogeneous or it has codimension two in $V$,
due to the theorem of Throbergsson \cite{Th2}.
In the former case, the maximal connected subgroup of $\OG V$
preserving~$S$ acts polarly
on $V$ and has $S$ as a principal orbit~\cite{PT}.
In the latter case, the cohomogeneity of $ G$ on $V$ is $3$.
Such actions are listed in~(i) above, and we can directly check that (b) holds in each
case by finding a larger connected
group acting as the isotropy representation of a symmetric
space of rank two, namely, $(\SO2\times\SO{16},\R^2\otimes_{\mathbf R}\R^{16})$,
$(\U 2\times\SU{2n},\C^2\otimes_{\mathbf C}\C^{2n})$, $(\SP2\times\SP n,\mathbf C^4\otimes_{\mathbf H}\C^{2n})$, respectively. Hence (b) and~(c) are equivalent.

Next we prove that (a) implies~(b). If $\tau$ is as in~(a), then the induced representation of $H^\circ$ is reducible, since $H^\circ$ is a torus.
Using~\cite[Theorem~1.7]{GL}, we deduce that 
the action of $H^\circ$ on $W$ can be identified with that of the maximal torus
$T^k$  of $\mathbf{SU}(k+1)$ on $\C^{k+1}$, for some $k\geq 1$
(possibly after replacing~$H$ and~$\tau$).  
Of course, $H$ is contained in the normalizer $N$ of $H^\circ$
in $\OG W$. Moreover, the connected component $N^\circ$ is the
maximal torus (of rank $k+1$)
of the unitary group of $W$, which acts polarly on $W$.

Let $X= V/G = W/H$, let $Y=X/(N/H)=W/N$ and denote by
$\pi$ denote the composite map
\[ V \to X  \to Y. \]
This composite map $\pi$ is a \emph{submetry} (see \cite{Ly,Guj}),
whose fibers are smooth
equidistant submanifolds of $V$.
Denote by $Y_{reg}$ the subset of principal $N/H$-orbits in $X$,
denote by $V_{reg}$ the set of $G$-regular points in $V$
and put $V'=\pi^{-1}(Y_{reg})\cap V_{reg}$. Then $V'$ is an open dense
subset of $V$ and $\pi:V'\to Y_{reg}$ is a Riemannian submersion.
Since the action of $N$ on $W$ is  polar,
$Y_{reg}$ is flat, so O'Neill's formula implies that the $\pi$-horizontal distribution
in $V'$ is integrable.  Hence the regular fibers of $f$ have flat normal bundles.
Moreover, each regular fiber is equifocal (the focal points are given by the intersection of a horizontal geodesic with singular fibers).
It follows that the fibers of $\pi$ in $V$
yield an isoparametric foliation $\mathcal F$
by full irreducible submanifolds of $V$.

We have $\cod\mathcal F=\dim Y = \dim X-1=k+1$.
In case $k=1$, the cohomogeneity of $\rho$ is $3$ and the result
follows as above, so we may assume $k\geq2$.
By the theorem of Thorbergsson~\cite{Th2}, $\mathcal F$ is homogeneous,
and the maximal connected  subgroup $\hat G \subset\mathbf O(V)$
which preserves the foliations acts transitively on the leaves.
By definition, $\hat G$ is closed, acts polarly and contains
$G$. By construction, $V/\hat G =Y$.
Hence we have~(b).

It was already remarked in~\cite{GL}
that the representations in~(i), (ii), (iii) all
satisfy~(a).  More precisely, representations of cohomogeneity three 
have copolarity one~(\cite[Example~1.9]{GL}; see also~\cite{str}).
The representations listed in (iii) have
copolarity 2 and 3 respectively~\cite[Theorem~1.11]{GL}. 
Thus, due to Theorem~1.5 in the same paper, the representations 
listed in~(i) and~(iii)  satisfy~(a).
That the representations coming from the Hermitian symmetric spaces and
listed in~(ii) satisfy (a) has been observed in~\cite[Example~1.10]{GL};
see also~\cite{got}.

We shall finish the proof of Theorem~\ref{main} in the next section
by proving that the representations satisfying~(b) are listed
in~(i), (ii) and~(iii) above.

\section{The classification}

Let $\hat G$ be as in~(b). Consider the maximal connected
subgroup~$K$ of~$\OG V$ with the same orbits as $\hat G$.
It is known that the action of $K$ on $V$ is the isotropy
representation of an irreducible symmetric space~\cite[Prop.~4.3.9]{BCO},
say, of compact type. Let $(L,K)$ and $M=L/K$ be the corresponding
symmetric pair and symmetric space, respectively.
The cohomogeneity $c(K,V)$ is just the rank~$r=\mathrm{rk}(M)$
of the symmetric space $M$. We have $G\subset K$ and
$c(G,V)=c(K,V)+1=r+1$. We run over the cases
of irreducible symmetric spaces of compact type~\cite{Hel,Wolf}.

\subsection{Adjoint representations}\label{II}
These are associated to symmetric spaces of type~II.
For the adjoint representation of $K$ on its Lie algebra $\Lk$
we have that $\Lg$, the Lie algebra of $G$,
is a proper invariant subspace of $\Lk$
for the $G$-action, hence the action of $G$ on $\Lk$
is not irreducible.

\subsection{List of symmetric spaces of type I}
One says that the symmetric space $M=L/K$ has maximal rank
if $\mathrm{rk}(M)=\mathrm{rk}(L)$. Such spaces are marked with $*$
below. The space marked $**$ has maximal rank if and only
if $|p-q|\leq1$.
\[ \begin{array}{|c|c|c|}
\hline
L/K & \textsl{dimension} & \textsl{rank} \\
\hline
\SU n/\SO n* & \mbox{\small$\frac12(n-1)(n+2)$} & \mbox{\small $n-1$}\\
\SU{2n}/\SP n  & \mbox{\small$(n-1)(2n+1)$} & \mbox{\small $n-1$} \\
\SU{p+q}/\mathbf S(\U p\times\U q) & 2pq & \mbox{\small $\min\{p,q\}$} \\
\SO{p+q}/(\SO p\times\SO q)** & pq & \mbox{\small $\min\{p,q\}$} \\
\SO{2n}/\U n & n(n-1) & [\frac n2] \\
\SP n/\U n* & n(n+1) & n \\
\SP{p+q}/(\SP p\times\SP q) & 4pq & \mbox{\small $\min\{p,q\}$} \\
\E6/\SP4* & 42 & 6 \\
\E6/\SU6\SU2 & 40 & 4 \\
\E6/\Spin{10}\U1 & 32 & 2 \\
\E6/\F & 26 & 2 \\
\E7/\SU8* & 70 & 7 \\
\E7/\Spin{12}\SU2 & 64 & 4 \\
\E7/\E6\U1 &  54 & 3 \\
\E8/\Spin{16}*& 128 & 8   \\
\E8/\E7\SU2 & 112 & 4  \\
\F/\SP3\SP1* & 28 & 4 \\
\F/\Spin9 & 16 & 1   \\
\G/\SO4* & 8 & 2 \\
\hline
\end{array} \]

\subsection{Spaces of low rank} Here we deal with the case $r\leq4$.
We rely on the classification of
irreducible representations of connected groups
of cohomogeneity at most five (\cite{str2} and \cite[Theorem~1.11]{GL}).

Recall that $c(G,V)=r+1$. If $r=1$, then $c(G,V)=2$ and $\rho$ is polar.

If $r=2$,
then $c(G,V)=3$ and $\rho$ is listed in~(i)~\cite{str2}.

If $r=3$, then $c(G,V)=4$
and $\rho$ is given in~\cite[Table~1]{GL}. The only cases in the
table not listed in~(ii) or~(iii) are the first two,
for which $\dim M=\dim V\leq8$ and
$\mathrm{rk}(M)=3$. But there are no symmetric spaces under these conditions.

Finally, if $r=4$ then $c(G,V)=5$ and $\rho$ is given in~\cite[Table~2]{GL}.
The only cases in the table not listed in~(ii) or~(iii)
are the first two and the last.
In those cases the dimension of $V$ is $8$, $12$ or $24$.
Going through the list of symmetric spaces of type~I  and  rank~$4$, we
see that $\dim (V)=\dim (M)$  cannot be $8$ or $12$ and we are left with
the 24-dimensional case
$M= \SO {10} / (\SO6 \times \SO4$).  Thus we only have to exclude the
case $K= \SO6 \times \SO4$ and $G=\U3 \times \SP2$.
However, $\U3 \times \SP2$ cannot be a subgroup of
$\SO6 \times \SO4$.

Henceforth we shall assume $r\geq5$.

\subsection{Spaces of maximal rank}
In this case, the principal isotropy group
$K_{\mathrm{princ}}$ is finite, so also $G_{\mathrm{princ}}$ is finite.
Therefore  $G$ has codimension one in $K$, so it is a normal
subgroup of $K$.  Since $K$ has a normal subgroup of codimension one,
it has a normal (hence central) subgroup of dimension $1$. Thus $K$ has a circle factor and the corresponding symmetric space $M$ is Hermitian.
Since the rank is maximal
and at least $5$,  we deduce $M=\SP n/\U n$.
In this case, there is a unique codimension one subgroup of $K$
and we get one example $(\SU n,S^2\C^n)$ listed in~(ii).

\subsection{Reformulation} \label{reformulation}
There remain five classical families
of symmetric spaces that we are going to analyse in the following.
For easy reference,
we list their isotropy representations together with the identity
components of their principal isotropy groups.

{\small
\setlength{\extrarowheight}{.08in}
\[ \begin{array}{|c|c|c|c|}
\hline
K & V & H:=(K_{\mathrm{princ}})^\circ & Conditions \\
\hline
\SP n  & [\Lambda^2\C^{2n}\ominus\C]_{\mathbf R} & \SP1^n & n\geq6\\
\mathbf S(\U p\times\U q)\ \mbox{($p\geq q$)}
& \C^p\otimes_{\mathbf C}\C^{q*} & \mathbf S(\U{p-q}\times\U1^q) & p\geq q\geq5\\
\SO p\times\SO q\ \mbox{($p\geq q$)} & \R^p\otimes_{\mathbf R}\R^q & \SO{p-q}
& p\geq q+2\geq7\\
\U n & \Lambda^2\C^n & \begin{array}{c}
  \SU2^{\frac n2}\ \mbox{if $n$ is even} \\
  \SU2^{\frac{n-1}2}\U1\ \mbox{if $n$ is odd}\end{array} & n\geq10\\
\SP p\times\SP q\ \mbox{($p\geq q$)} & \Q^p\otimes_{\mathbf H}\Q^{q*}&
\SP{p-q}\times\SP1^q & p\geq q\geq5\\
\hline
\end{array} \]}

\subsection{Relation with the work of Kollross} \label{Kollross}
Since $G$ acts on the Euclidean space~$V$ with cohomogeneity one less than $K$,
the action of $G$ on the principal orbits of $K$ has cohomogeneity one.
Since those orbits are of type $K/K_{\mathrm{princ}}$,
this means that   $G\times K_{\mathrm{princ}}$,  and therefore $G\times H$
where $H:=K_{\mathrm{princ}}^\circ$, acts on the Lie group $K$
with cohomogeneity one, where the first factor acts
from the left and the second from the right.
Such (and more general) cohomogeneity one actions
have been extensively studied by Kollross.  He has
classified such actions in the case of simple groups $K$,
unfortunately, only  \emph{up to orbit equivalence}.  More precisely,
he has shown that for any such action, one finds larger connected groups
$G \subset \tilde G \subset K$ and $H \subset \tilde H \subset K$,
such that the triple $(\tilde G, K, \tilde H)$ is listed in \cite[Theorem~B]{K}.

Using this theorem of Kollross  it will be easy to finish the proof of our
theorem. We only need to circumvent some minor difficulties related to the
fact that our group $K$ needs not be simple and
that our pair $(G,H)$ needs not be maximal.

We will derive from~\cite{K} the following:

\begin{lem}  \label{genlem}
Let $\mathbf F$ be $\mathbf R$, $\mathbf C$ or $\mathbf H$.
Let $K$ be respectively  $\SO n$, $\SU n$, $\SP n$ acting on
$V=\mathbf F^n$, with $n\geq 5$.
Let $H$ be a subgroup of $K$. Assume that $V$ decomposes into $H$-invariant
subspaces as $V=V_1 \oplus V_2 \oplus V_3$, with
$1\leq \dim _{\mathbf F} (V_1 ) \leq \dim _{\mathbf F} (V_2) \leq \dim _{\mathbf F} (V_3 )$.
Assume that for a proper subgroup $G\subset K$ the group $G\times H$ acts
with cohomogeneity at most one on $K$. Then
$\mathbf F \neq \mathbf H$.  If $\mathbf F=\mathbf C$ then $n$ is even and
$G$ is contained in $\SP {\frac n 2}$.
If $K =\SO 7$ then $G \subset \mathbf G _2$. Finally, in all cases,
$V_3$ has either $\mathbf F$-dimension or $\mathbf F$-codimension at most $3$.
\end{lem}

\begin{proof}

Assume first that $G$ acts reducibly on $V$, namely, leaves a proper subspace
$W$ of $V$ invariant. Let $K_W =\{k \in K | k \cdot W =W \}$.
Then  $\mathcal O=K/K_W$ is the Grassmannian of $\dim (W)$-dimensional subspaces of
$V$ and, since $G\subset K_W$, the group $H$ acts with cohomogeneity at most one
on~$\mathcal O$. However, this is impossible, since $H$ preserves the relative
positions to the $V_i$ of any subspace $W' \in \mathcal O$
and this implies that $\mathcal O/H$ is at least two-dimensional.

Thus $G$ acts irreducibly on $V$.  Consider now a maximal connected
subgroup $\tilde G$ of $K$ containing $G$. Set $\tilde H=(K_{V_3})^\circ$.
Then $\tilde H$ is a maximal connected subgroup of $K$ containing $H$.
Thus $\tilde G \times \tilde H$ acts on $K$ with cohomogeneity
at most one. Hence it must appear either in the list of Onishik~\cite{on},
or in the list of Kollross.

From the list of Onishchik, we only get $K=\SO 7$ and $\tilde G =G_2$.
In the list of Kollross, we see the remaining statements.
\end{proof}

\subsection{General consequences}
In
case $K$ is a direct product $K=K_1\times K_2$, we denote
by $G_i$ and $H_i$, respectively, the projections of $G$ and of
$H$ to $K_i$.  Then
$1=\dim (K/G\times H) \geq\dim (K_1 /G_1 \times H_1 ) +\dim (K_2 /G_2 \times H_2)$.

In any case, observe now that for each one of the groups $K$ appearing in
Subsection~\ref{reformulation}, the assumption that the rank is at least~$5$
and the structure of the principal isotropy group $K_{\mathrm{princ}}$
implies the following: for any simple non-abelian factor
$K_i$ of $K$, the triple $(G_i,K_i,H_i)$
has the form as in \lref{genlem} above.
Moreover, the decomposition of the space $\mathbf F^n$
into $H_i$-irreducible subspaces has at least five
summands.  Next we proceed case by case.

\subsection{Symplectic cases}
Assume that $K=\SP n$ or $K= \SP p  \times \SP q$.  From \lref{genlem} we see that
the projection of $G$ to each factor $K_i$ coincides with $K_i$.  We  deduce
$K=G$, unless possibly $K=\SP q \times \SP q$ and $G$ being the diagonal subgroup
of $K$ (up to conjugation), isomorphic to~$\SP q$.
However, in the last case $H = \SP 1 ^q$  and $G\times H$ cannot act on $K$ with
codimension~$1$ since $\dim (K) -\dim (G) -\dim (H) >1$.

\subsection{Real case}
Assume that $K=\SO p \times \SO q$, with $p\geq q+2 \geq 7$.
The projection  of $H$ to
$\SO q$ is trivial, thus the projection of $G$ to $\SO q$ has codimension at
most $1$ in $\SO q$.
Hence this projection coincides with $\SO q$.  The projection of $H$ to $\SO p$
fixes pointwise a $5$-dimensional
subspace.  If $p \geq 8$ we find an $H$-invariant decomposition
$\mathbf R ^p =V_1 \oplus V_2 \oplus V_3$ with dimension
and codimension of $V_3$ at least $4$; and applying \lref{genlem} we deduce that
the projection $G_1$ of $G$ to $\SO p$
is $\SO p$.   We are left with the case $p=7, q=5$.
Then the projection $H_1$ of $H$ to $\SO p$ is $\SO 2$
and, due to \lref{genlem}, the group $G_1$ must be a subgroup of the
$14$-dimensional group $\mathbf G_2$.
Then $G_1 \times H_1$ cannot act on $\SO 7$ with cohomogeneity $1$.

It follows that the projections of $G$ to the simple factors of $K$ coincide with the factors. Since by assumption $p\neq q$, we get  $G=K$.

\subsection{Complex case}
Assume that $K= \U n$ with $n\geq 10$.  Then, for the projection $H_1$ of $H$
to the semisimple part $\SU n$, the space $\mathbf C ^n$ decomposes into a
sum of at least five $H_1$-invariant subspaces of complex dimension $2$.
From \lref{genlem} we deduce that the projection $G_1$ of $G$ to $\SU n$ coincides
with $\SU n$. Therefore $G$ has codimension $1$ in $K$. Hence~$G= \SU n$
and we get one example $(\SU n,\Lambda^2\C^n)$, where $n$ is even, listed in~(ii)
(note that in case $n$ is odd, the action of $\SU n$ is orbit-equivalent
to that of $\U n$, which is polar~\cite{EH}).

\subsection{Complex Grassmanian}
Assume that $K= \mathbf S (\U p \times \U q )$ with $p\geq q \geq 5$.
Applying \lref{genlem} to the projection  $H_2$ of $H$ and $G_2$ of $G$
to $\SU q$, we see that $G_2 =\SU q$, unless possibly $q= 6$ and
$G_2 \subset \SP 3$. But in  the latter case, $H_2= \mathbf S (\U 1^6) $ and
$G_2 \times H_2$ cannot act with cohomogeneity $1$ on $\SU 6$.

Similarly, for the projection of $G_1$ of $G$ to $\SU p$ we deduce that
$G_1 = \SU p$, unless possibly $p=6$ and $G_1 \subset \SP 3$.
In the latter case the projection $H_1$ of $H$ to $\SU p$ is equal
to $\mathbf S (\U 1^6 ) $, which is again impossible by dimensional reasons.

Thus the projection of $G$ to the simple factors of $K$ is surjective.
Now either the projection $\bar G$ of $G$ to $\SU p \times \SU q$ is surjective,
or $p=q$ and the image $\bar G$ is the twisted diagonal subgroup
$\{(g,\alpha(g)):g\in\SU q\}$ where $\alpha$ is an automorphism of
$\SU q$.
The last case is again impossible by dimensional reasons.
Thus $\bar G = \SU p \times \SU q$.

Therefore, $G$ has codimension one in $K$. Hence $G= \SU p \times \SU q$
and we get one example $(\SU p\times\SU p,\C^p\otimes\C^p)$ listed in~(ii)
(note that in case $p>q$ the action of $\SU p\times \SU q$ is orbit-equivalent to
that of $\mathbf S(\U p\times \U q)$, which is polar~\cite{EH}).

%\bibliographystyle{amsalpha}
%\bibliography{ref}

\providecommand{\bysame}{\leavevmode\hbox to3em{\hrulefill}\thinspace}
\providecommand{\MR}{\relax\ifhmode\unskip\space\fi MR }
\providecommand{\MRhref}[2]{%
  \href{http://www.ams.org/mathscinet-getitem?mr=#1}{#2}
}
\providecommand{\href}[2]{#2}

\end{document}